\documentclass[a4paper]{amsart}
\usepackage{amscd}
\usepackage{amssymb}
\usepackage{color}
\usepackage[T1]{fontenc}
\usepackage[utf8]{inputenc}
\usepackage{hyperref}
\usepackage[arrow,curve,matrix]{xy}

\usepackage{mathpazo}
\usepackage{mathrsfs}

\usepackage{graphicx}
\usepackage{color}

\DeclareFontFamily{OMS}{rsfs}{\skewchar\font'60}
\DeclareFontShape{OMS}{rsfs}{m}{n}{<-5>rsfs5 <5-7>rsfs7 <7->rsfs10 }{}
\DeclareSymbolFont{rsfs}{OMS}{rsfs}{m}{n}
\DeclareSymbolFontAlphabet{\scr}{rsfs}

\newcommand{\sF}{\scr{F}}
\newcommand{\sG}{\scr{G}}

\newcommand{\sJ}{\scr{J}}

\newcommand{\sO}{\scr{O}}

\newcommand{\bC}{\mathbb{C}}

\newcommand{\bN}{\mathbb{N}}

\DeclareMathOperator{\Jet}{Jet}
\DeclareMathOperator{\codim}{codim}
\DeclareMathOperator{\Hamilton}{Hamilton}
\DeclareMathOperator{\Hom}{Hom}
\DeclareMathOperator{\Image}{Image}
\DeclareMathOperator{\Lie}{L}

\DeclareMathOperator{\Spec}{Spec}

\theoremstyle{plain}    
\newtheorem{thm}{Theorem}[section]
\newtheorem{defn}[thm]{Definition}
\newtheorem{setup}[thm]{Setup} 
\numberwithin{equation}{thm}
\numberwithin{figure}{section}
\theoremstyle{plain}    
\newtheorem{cor}[thm]{Corollary}
\newtheorem{lem}[thm]{Lemma}
\newtheorem{fact}[thm]{Fact}

\theoremstyle{plain}    
\newtheorem{prop}[thm]{Proposition}
\theoremstyle{remark}
\newtheorem{rem}[thm]{Remark}
\newtheorem{subclaim}[equation]{Claim} %%Delete [...] to re-start numbering
\newtheorem{notation}[thm]{Notation}
\newtheorem{example}[thm]{Example}

\definecolor{tomato}{RGB}{180,62,39}
\definecolor{forrest}{RGB}{81,133,49}
\definecolor{lighttomato}{RGB}{253,65,65}
\definecolor{lightforrest}{RGB}{145,237,87}
\definecolor{mygreen}{RGB}{40,104,69}
\definecolor{mygreen2}{RGB}{3,149,39}
\definecolor{darkolivegreen}{RGB}{102,118,75}
\definecolor{cranegreen}{RGB}{102,118,75}
\definecolor{mydarkblue}{RGB}{10,92,153}
\definecolor{myblue}{RGB}{57,222,186}
\definecolor{pinkish}{RGB}{213,83,222}
\definecolor{colD}{RGB}{213,83,222}
\definecolor{defb}{RGB}{213,83,222}
\definecolor{goldenrod}{RGB}{225,115,69}
\definecolor{mauve}{RGB}{224, 176, 255}
\definecolor{fuchsia}{RGB}{255, 0, 255}
\definecolor{lavender}{RGB}{230, 230, 250}
\definecolor{gold}{RGB}{255, 215, 0}
\definecolor{orange}{RGB}{255, 127, 0}
\definecolor{maroon}{RGB}{123, 17, 19}
\definecolor{brightmaroon}{RGB}{195, 33, 72}
\definecolor{richmaroon}{RGB}{176, 48, 96}
\definecolor{green}{RGB}{3,149,39}

\newcommand{\PreprintAndPublication}[2]{#1}

\date{\today}

\author{Clemens J\"order}
\address{Clemens J\"order, Mathematisches Institut, Albert-Ludwigs-Universitt
  Freiburg, Eckerstraße 1, 79104 Freiburg im Breisgau, Germany}
\email{\href{mailto:clemens.joerder@math.uni-freiburg.de}{clemens.joerder@math.uni-freiburg.de}}

\author{Stefan Kebekus}
\address{Stefan Kebekus, Mathematisches Institut, Albert-Ludwigs-Universitt
  Freiburg, Eckerstraße 1, 79104 Freiburg im Breisgau, Germany}
\email{\href{mailto:stefan.kebekus@math.uni-freiburg.de}{stefan.kebekus@math.uni-freiburg.de}}
\urladdr{\url{http://home.mathematik.uni-freiburg.de/kebekus}}

\thanks{Clemens Jörder and Stefan Kebekus were supported in part by the
  DFG-Forschergruppe 790 ``Classification of Algebraic Surfaces and Compact
  Complex Manifolds''.}

\title{Deformation along subsheaves, II}

\begin{document}

\begin{abstract}
  Let $f : Y \to X$ be the inclusion map of a compact reduced subspace of a
  complex manifold, and let $\sF \subseteq T_X$ be a subsheaf of the tangent
  bundle which is closed under the Lie bracket, but not necessarily a sheaf of
  $\sO_X$-algebras. This paper discusses criteria to guarantee that
  infinitesimal deformations of $f$ which are induced by $\sF$ lift to
  positive-dimensional deformations of $f$, where $f$ is deformed ``along the
  sheaf $\sF$''.

  In case where $X$ is complex-symplectic and $\sF$ the sheaf of locally
  Hamiltonian vector fields, this partially reproduces known results on
  unobstructedness of deformations of Lagrangian submanifolds. The proof is
  rather elementary and geometric, constructing higher-order liftings of a
  given infinitesimal deformation using flow maps of carefully crafted
  time-dependent vector fields.
\end{abstract}

\maketitle
\tableofcontents

\section{Introduction and main result}

\subsection{Introduction}

Let $f: Y \to X$ be the inclusion map of a compact reduced subspace of a
complex manifold. We aim to deform the morphism $f$, keeping the complex
spaces $X$ and $Y$ fixed. For this purpose let us fix a first-order
infinitesimal deformation of $f$, say $\sigma \in H^0\bigl(Y ,\, f^*T_X
\bigr)$ --- we refer to the earlier paper \cite[Sect.~1]{KKL10} for a
discussion of infinitesimal deformations, and for other notions used here. We
ask for conditions to guarantee that $\sigma$ is effective. In other words, we
ask for conditions that guarantee the existence of a disk $\Delta \subset
\mathbb C$, centered about $0$, and a family of morphisms,
$$
\begin{array}{rccc}
  F : &  \Delta \times Y & \to & X\\
  & (t,y) & \mapsto & F_t(y)
\end{array}
$$
such that $F_0 = f$ and such that the infinitesimal deformation induced by the
family, $\sigma_{F,0} := \frac{d}{d t}\bigr|_{t=0}F_t \in H^0\bigl(Y ,\,
f^*T_X \bigr)$, agrees with $\sigma$. In case when $Y$ is a manifold, the most
general result in this direction is due to Horikawa.

\begin{thm}[Horikawa's criterion,  \cite{MR0322209}]\label{thm:hori}
  If $H^1\bigl(Y ,\, f^*T_X \bigr)$ vanishes, then any first-order
  infinitesimal deformation of $f$ is effective. \qed
\end{thm}

Vanishing of the obstruction space $H^1\bigl(Y ,\, f^*T_X \bigr)$ is a
sufficient, but not a necessary condition for the existence of liftings. In
settings where the geometry of the target manifold $X$ is well-understood, it
is often possible to prove existence of liftings even if $H^1 \bigl( Y,\, f^*
T_X \bigr)$ is large.

\subsubsection*{Earlier results}

One situation where effectivity of infinitesimal deformations can sometimes be
shown has been studied in \cite{KKL10}. The authors of~\cite{KKL10} considered
a coherent subsheaf $\sF \subseteq T_X$ of $\sO_X$-modules, closed under
Lie-bracket, and an infinitesimal deformation induced by $\sF$,
$$
\sigma \in H^0 \left( Y ,\, \Image f^* \sF \to f^* T_X \right).
$$
It was shown in \cite[Thm.~1.5]{KKL10} that $\sigma$ lifts to a family of
morphisms if the space $H^1 \bigl( Y,\, \Image f^* \sF \to f^* T_X \bigr)$
vanishes. Examples of sheaves $\sF$ that appear this way include (singular)
foliations, logarithmic tangent sheaves, or the tensor product of $T_X$ with
the ideal sheaf of a subvariety.

\subsubsection*{Result of this paper}

This paper is concerned with the case where the infinitesimal deformation
$\sigma$ is induced by a sheaf $\sF \subseteq T_X$ which is closed under
Lie-bracket, but is not necessarily a sheaf of $\sO_X$-modules. Examples are
given by sheaves of Hamiltonian vector fields on complex-symplectic manifolds,
or more generally sheaves of vector fields whose flows stabilize a given
tensor. The main result, formulated in Theorems~\ref{thm:main} and
\ref{thm:mainaddon}, generalizes and improves on \cite[Thm.~1.5]{KKL10}. In
case where $X$ is complex-symplectic, and $Y \subset X$ is a Lagrangian
submanifold, this reproduces results of Ran, Voisin, Kawamata, and others.

\subsubsection*{Aim and scope}

Written for the IMPANGA Lecture Notes series, this paper aims at simplicity
and clarity of argument. It does not strive to present the shortest proofs or
most general results available. While everything said here can also be deduced
from the abstract machinery of deformation theory, we argue in a rather
elementary and geometric manner, constructing higher-order liftings of a given
infinitesimal deformation using flow maps of carefully crafted time-dependent
vector fields.

\subsection{Main result}

In order to formulate our result we start off with the necessary notation.

\begin{notation}
  If $X$ is any complex manifold, denote the sheaf of locally constant
  functions on $X$ by $\mathbb C_X \subseteq \sO_X$.
\end{notation}

Throughout the present paper, we will frequently consider subsheaves $\sF
\subseteq T_X$, such as the sheaf of Hamiltonian vector fields on a
complex-symplectic space, which are invariant under multiplication with
constants, but not necessarily under multiplication with arbitrary regular
functions. We call such $\sF$ a \emph{sheaf of $\mathbb C_X$-modules}. The
main results of this paper consider the following setup.

\begin{setup}\label{setup:main}
  Let $X$ be a complex manifold, and $Y \subseteq X$ a reduced complex
  subspace of $X$, with inclusion map $f: Y \to X$.  Further, let $\sF
  \subseteq T_X$ be a subsheaf of $\mathbb C_X$-modules, closed under Lie
  bracket.
\end{setup}

We aim to deform the inclusion map ``along the subsheaf $\sF$''. The following
two notions help to make this precise.

\begin{defn}[Infinitesimal deformations locally induced by $\sF$]\label{del:sidliF}
  In Setup~\ref{setup:main}, set
  $$
  \sF_Y := \Image \left( f^{-1} \sF \to f^* T_X \right) \subseteq f^*
  T_X.
  $$
  We call $\sF_Y$ the sheaf of \emph{infinitesimal deformations locally
    induced by $\sF$}. Sections $\sigma \in H^0 \bigl( Y,\, \sF_Y \bigr)$ are
  called \emph{infinitesimal deformations of $f$ which are locally induced by
    $\sF$}.
\end{defn}

\begin{defn}[Obstruction sheaves]\label{def:obstrShf}
  In Setup~\ref{setup:main}, we call a subsheaf $\sG \subseteq \sF_Y$ an
  \emph{obstruction sheaf for $\sF$} if the Lie-bracket of any two vector
  fields in $\sF$ which agree along the image $f(Y)$ induces a section of
  $\sG$.

  More precisely, $\sG \subseteq \sF_Y$ is called an obstruction sheaf for
  $\sF$ if for any open set $U \subseteq X$ and any two vector fields $\vec
  A_1, \vec A_2 \in \sF(U)$ satisfying $f^* \vec A_1 = f^* \vec A_2$, the
  preimage of the Lie-bracket is contained in $\sG$,
  $$
  f^*\, [\vec A_1, \vec A_2] \in \sG\bigl( f^{-1}(U) \bigr).
  $$
\end{defn}

\begin{rem}
  Obstruction sheaves are generally not unique.  Example~\ref{ex:foliation}
  discusses a situation where there is more than one sheaf satisfying the
  requirements of Definition~\ref{def:obstrShf}.
\end{rem}

The main result of our paper essentially says that Setup~\ref{setup:main}, any
infinitesimal deformation of $f$ which is locally induced by $\sF$ is
effective if there exists an obstruction sheaf whose first cohomology group
vanishes.

\begin{thm}[Existence of deformations]\label{thm:main}
  In Setup~\ref{setup:main}, assume that $Y$ is compact and assume that there
  exists an obstruction sheaf $\sG \subseteq \sF_Y$ for $\sF$ such that $H^1
  \bigl(Y,\, \sG \bigr) = 0$.  Then any infinitesimal deformation $\sigma \in
  H^0\bigl( Y,\, \sF_Y \bigr)$ locally induced by $\sF$ is effective.  In
  other words, there exists an open neighbourhood $\Delta$ of $0 \in \bC$, and
  a family of morphisms
  $$
  \begin{array}{rccc}
    F : & \Delta \times Y & \to & X\\
    & (t,y) & \mapsto & F_t(y)
  \end{array}
  $$
  such that $F_0 = f$ and such that the infinitesimal deformation induced by
  the family, $\sigma_{F,0} := \frac{d}{d t}\bigr|_{t=0}F_t \in H^0\bigl(Y ,\,
  f^*T_X \bigr)$, agrees with $\sigma$.
\end{thm}

Understanding that the formulation of Theorem~\ref{thm:main} may sound a
little technical, we end the present subsection with two examples. In
Section~\ref{ssec:defalF} we will then see that the family of morphisms whose
existence is guaranteed by Theorem~\ref{thm:main} can often be chosen in a way
that geometrically relates to the sheaf $\sF$.

\begin{example}[Complex-symplectic manifolds]\label{ex:symp}
  Let $(X, \omega)$ be a complex-symplectic manifold and $\sF \subset T_X$ the
  sheaf of Hamiltonian vector fields. Let $Y \subseteq X$ be any compact
  complex submanifold, with inclusion map $f: Y \to X$. We assume that $Y$ is
  Kähler and consider the sheaf $T_Y^\perp \subseteq f^* T_X$ of vector fields
  which are perpendicular to $Y$, with respect to the symplectic form
  $\omega$. In other words, $T_Y^\bot$ is the sheaf associated to the presheaf
  $$
  U\mapsto \left\{ \vec A \in (f^*T_X)(U) \,\,\bigl|\,\, (f^* \omega)(\vec A,
    df(\vec B)) = 0 \text{ for all } \vec B \in T_Y(U) \right\},
  $$
  where $U\subset Y$ runs over the open subsets of $Y$. We will show in
  Section~\ref{sec:examples} that $H^0 \bigl(Y,\, \sF_Y \bigr) = H^0
  \bigl(Y,\, f^* T_X \bigr)$, and that $T_Y^\perp$ is an obstruction sheaf in
  this setting. Theorem~\ref{thm:main} thus asserts that any infinitesimal
  deformation $\sigma \in H^0 \bigl( X,\, f^*T_X \bigr)$ always lifts to a
  holomorphic family $F$ of morphisms if the cohomology group $H^1\bigl( Y,\,
  T_Y^\perp \bigr)$ vanishes.
\end{example}

\begin{rem}[Deformations of Lagrangian submanifolds]
  In the setting of Example~\ref{ex:symp}, if $Y \subset X$ is a Lagrangian
  submanifold, then $T_Y^\perp = T_Y$. We obtain that vanishing of $H^1\bigl(
  Y,\, T_Y \bigr)$, the tangent space to the Kuranishi-family of deformations
  of $Y$, is the only obstruction to lifting a given infinitesimal
  deformation. This partially reproduces results of Ran, Kawamata and Voisin
  on the unobstructedness of deformations of Lagrangian subvarieties,
  cf.~\cite{RAN92MR1487238MR1487238MR1487238, voisin-lagrangian, MR1144434,
    MR1487238} and the references there.
\end{rem}

\begin{example}[Deformation along a foliation]\label{ex:foliation}
  Let $X$ be a complex manifold and $\sF \subseteq T_X$ a (regular) foliation,
  i.e., a sub-vectorbundle of $T_X$ which is closed under Lie bracket. Again,
  let $Y \subseteq X$ be any compact submanifold of $X$. Let $T \subset Y$ be
  the set of points where the foliation is tangent to $Y$,
  $$
  T := \{ y \in Y \,|\, \sF|_y \subseteq T_Y|_y \}.
  $$
  It is clear that $\sF_Y = \sF|_Y$. We will show in Section~\ref{ssec:exDefF}
  that any of the two sheaves
  $$
  \sG_1 := \sF_Y \quad\text{and}\quad \sG_2 := \sF_Y \otimes \sJ_T
  $$
  are obstruction sheaves in this setting. Thus, if either $H^1\bigl( Y,\,
  \sF_Y \bigr)$ or $H^1\bigl( Y,\, \sF_Y \otimes \sJ_T \bigr)$ vanishes, then
  any infinitesimal deformation $\sigma \in H^0 \bigl( Y,\, \sF|_Y \bigr)$
  lifts to a holomorphic family of morphisms.
\end{example}

\subsection{Deformations along $\sF$}
\label{ssec:defalF}

Although the sheaf $\sF$ appears in the assumptions of Theorem~\ref{thm:main},
its conclusion seems to disregard $\sF$ entirely, as the family of morphisms
obtained in Theorem~\ref{thm:main} need not be related to $\sF$ in any way.
However, in all the examples we have in mind, there is a way to construct a
family of morphisms $F: \Delta \times Y \to X$ that relates to $\sF$
geometrically.

\subsubsection{Notation concerning higher-order infinitesimal deformations}

For a precise formulation of this result, we need to discuss locally closed
subspaces of the Douady-space of all holomorphic maps $Y \to X$ which
parametrize deformations \emph{along $\sF$}. The following notions concerning
\emph{higher-order} infinitesimal deformations of $f$ will be used in the
definition.

\begin{defn}[\protect{Higher-order infinitesimal deformations, cf.~\cite[Def.~2.12]{KKL10}}]\label{def:hoid}
  Let $f: Y \to X$ be a morphism from a complex space $Y$ to a complex
  manifold $X$. An \emph{$n$-th order infinitesimal deformation of $f$} is a
  morphism
  $$
  f_n : \Spec\, \bC[t]/t^ {n+1} \times Y \to X
  $$
  whose restriction to $Y \cong \Spec\, \bC \times Y$ agrees with $f$.
\end{defn}

The universal property of the Douady-space immediately yields another,
equivalent, definition of higher-order infinitesimal deformations.

\begin{notation}[\protect{Douady-space of morphisms, cf.~\cite[Sect.~2]{MR1326625}}]
  Let $X$, $Y$ be complex spaces and assume that $Y$ is compact. We denote the
  Douady-space of morphisms from $Y$ to $X$ by $\Hom\bigl( Y,\, X\bigr)$.
\end{notation}

\begin{fact}[Higher-order deformations as morphisms to the Douady-space]
  In the setting of Definition~\ref{def:hoid}, assume additionally that $Y$ is
  compact. To give an $n$-th order infinitesimal deformation of $f$, it is
  then equivalent to give a morphism
  $$
  f_n : \Spec\, \bC[t]/t^ {n+1} \to \Hom\bigl(Y,\,X\bigr)
  $$
  whose closed point maps to the point $[f] \in \Hom\bigl(Y,\,X\bigr)$
  representing the morphism $f$. \qed
\end{fact}

Alternatively, any $n$-th order deformation of $f$ can also be seen as a
section in the pull-back of the $n$-th order jet-bundle of $X$. Jet bundles
and their fundamental properties are reviewed in Section~\ref{ssec:jetbundles}
below.

\begin{fact}[\protect{Higher-order deformations as sections in $\Jet^n$, \cite[Prop.~2.13]{KKL10}}]\label{fact:hoddJ}
  In the setting of Definition~\ref{def:hoid}, to give an $n$-th order
  infinitesimal deformation of $f$, it is equivalent to give a section $f_n :
  Y \to f^* \Jet^n(X)$. \qed
\end{fact}

\subsubsection{Spaces of deformations along $\sF$}

Sections in $\Jet^n(X)$, which appear in the description of higher-order
deformations given in Fact~\ref{fact:hoddJ}, can be constructed using flows of
(time-dependent) vector fields on $X$. We will recall in
Section~\ref{ssec:tdinducedjets} that if $U \subseteq X$ is any open set, if
$\vec A$ is any time-dependent vector field on $U$, and $n$ any number, then
the flow of $\vec A$ induces a section $\tau^n_{\vec A} : U \to \Jet^n(U)$. We
call higher-order infinitesimal deformations of $f$ to be \emph{induced by
  time-dependent vector fields in $\sF$} if they locally arise in this way.

\begin{defn}[Deformations induced by time-dependent vector fields in $\sF$]\label{defn:infdefoindbyf}
  In Setup~\ref{setup:main}, an $n$-th order deformation $f_n : Y \to f^*
  \Jet^n(X)$ of $f$ is said to be \emph{locally induced by time-dependent
    vector fields in $\sF$}, if there exists a cover of $Y$ by open subsets of
  $X$, say $Y \subseteq \cup_{\alpha \in A} U_\alpha$, and for any $a \in A$
  time-dependent vector fields of the following form,
  $$
  \vec A_a = \sum_{i=0}^n t^i \vec A_{a,i} \quad \text{where}\quad
  \vec A_{a,i} \in \sF(U_\alpha)
  $$
  such that $f_n|_{U_a\cap Y} = \tau^n_{\vec A_a}|_{U_a\cap Y}$, for all $a
  \in A$.
\end{defn}

\begin{defn}[Spaces of deformations along $\sF$]\label{defn:spaceofdefo}
  In Setup~\ref{setup:main}, assume additionally that $Y$ is compact. A
  locally closed analytic subspace $H\subset\Hom(Y,X)$ containing $[f]$ is
  called \emph{space of deformations along $\sF$}, if for any infinitesimal
  deformation $f_n$ of $f$ that is locally induced by time-dependent vector
  fields in $\sF$, the corresponding morphism $\Spec\,\bC[t]/t^{n+1} \to
  \Hom(Y,X)$ factors through $H$.
\end{defn}

\begin{example}[Complex-symplectic manifolds]
  In the setting of Example~\ref{ex:symp}, where $(X, \omega)$ is a
  complex-symplectic manifold, set
  $$
  H := \{ g \in \Hom(Y,X) \,|\, g^*\omega = f^*\omega\} \subseteq
  \Hom(Y,X).
  $$
  Since (time-dependent) Hamiltonian vector fields preserve the symplectic
  form $\omega$, it is clear that $H$ is a space of deformations along $\sF$.
\end{example}

\begin{example}[Foliated manifolds]
  In the setting of Example~\ref{ex:foliation}, where $\sF \subseteq T_X$ is a
  foliation, the existence of a space of deformations along $\sF$ with very
  good properties has been shown in \cite[Cor.~5.6]{KKL10}.
\end{example}

\subsubsection{Existence of deformations along $\sF$}

In cases where a space of deformations along $\sF$ exists, the deformation
family constructed in Theorem~\ref{thm:main} can be chosen to factor via that
space. This complements and strengthens Theorem~\ref{thm:main} in our special
situation.

\begin{thm}[Existence of deformations along $\sF$, strenghtening of Theorem~\ref{thm:main}]\label{thm:mainaddon}
  In the setting of Theorem~\ref{thm:main}, assume in addition that there
  exists a space $H \subseteq \Hom \bigl(Y,\,X \bigr)$ of deformations along
  $\sF$. Then there exists a family $F : \Delta \times Y \to X$ such that $F$
  satisfies all properties stated in Theorem~\ref{thm:main}, and such that the
  associated map $F : \Delta \to \Hom\bigl(Y,\,X\bigr)$ factors through $H$.
\end{thm}

\subsection{Outline of this paper, acknowledgements}

\subsubsection{Outline}

Section~\ref{sec:jets} summarizes fundamental facts concerning jet bundles
associated with a complex manifold that will be needed in the
sequel. Subsections~\ref{ssec:jetbundles} and~\ref{ssec:tdvf} review the
notions of jet bundles and time-dependent vector fields, respectively. The
sections in the jet bundles arising from flow maps associated with
time-dependent vector fields are discussed in the subsequent
subsection~\ref{ssec:tdinducedjets}.

Section~\ref{sec:atdhvf} is the technical core of the paper. In
Setup~\ref{setup:main} the choice of an obstruction sheaf leads to the notion
of \emph{admissible} time-dependent vector fields on $X$. The rather technical
definition of these time-dependent vector fields is justified by the
properties of the jets induced by them, as it is formulated in
Subsection~\ref{ssec:jibatdhvf} and proven in the remainder of
Section~\ref{sec:atdhvf}.

The local definition and properties of admissible time-dependent vector fields
and the induced jets are globalized in Section~\ref{ssec:higherorderadm},
yielding the notion of admissible higher-order infinitesimal deformations of
the inclusion $f:Y\to X$.  This leads us to a lifting criterion for
infinitesimal higher-order deformations of $f$ based on the cohomology
vanishing assumption that appears in the statement of Theorem~\ref{thm:main}.

Once the result in Section~\ref{ssec:higherorderadm} is established, the
actual proof of Theorem~\ref{thm:main} and Theorem~\ref{thm:mainaddon} is a
short argument that is outlined in Section~\ref{sec:mainproof}.

The paper concludes with a detailed review of the two examples mentioned so
far, namely the deformation on complex-symplectic manifolds and on foliated
manifolds in Sections~\ref{sec:examples} and~\ref{ssec:exDefF}, respectively.

\subsubsection{Acknowledgements}

A first version of the main results appeared in the diploma thesis of Clemens
Jörder, \cite{Joerder10}, supervised by Stefan Kebekus.  Work on the project
was initiated by discussions between Stefan Kebekus and Jarosław A.~Wiśniewski
that took place during the 2009 MSRI program in algebraic geometry. Both
authors would also like to thank Jun-Muk Hwang for numerous discussions on the
subject.

This paper was written as a contribution for the proceedings of the 2010
IMPANGA summer school. The authors thank the organizers of that event.

\section{Jet theory}
\label{sec:jets}

The proof of Theorem~\ref{thm:main} uses the convenient language of jet
bundles. Sections~\ref{ssec:jetbundles} and~\ref{ssec:tdvf} summarize basic
facts and definitions about jet bundles and time-dependent vector fields,
respectively.  Section~\ref{ssec:tdinducedjets} contains an important formula
concerning the jets induced by flow maps of time-dependent vector fields. A
more detailed introduction to jets is found in \cite[Sect.~2]{KKL10} and the
references quoted there.

\subsection{Jet bundles}\label{ssec:jetbundles}

Jet bundles generalize the notion of tangent bundles. If $X$ is any complex
manifold, an $n$-th order jet on $X$ is an equivalence class of curve germs,
where two germs are considered equivalent if they agree to $n$-th order. A
precise definition is given as follows.

\begin{defn}[Jets on complex manifolds]\label{defn:jetsonmanifold}
  Let $X$ be any complex manifold and $x \in X$ a point. An \emph{$n$-th order
    jet} at $x\in X$ is a morphism
  $$
  \sigma: \Spec\, \bC[t]/t^{n+1} \to X
  $$
  of complex spaces which maps the closed point of $\Spec\, \bC[t]/t^{n+1}$ to
  $x$.
\end{defn}

\begin{defn}[Jet bundle]
  If $X$ is a complex manifold, define the \emph{$n$-th order jet bundle of
    $X$} as
  $$
  \Jet^n(X) := \Hom \bigl( \Spec\, \bC[t]/t^{n+1},\, X \bigr).
  $$
  As a set, $\Jet^n(X)$ equals the set of $n$-th order jets on $X$.
\end{defn}

Recall from Fact~\ref{fact:hoddJ} that higher-order infinitesimal deformations
$f_n : \Spec\, \bC[t]/t^{n+1} \times Y \to X$ of $f$ can be considered as
$X$-morphisms $f_n:Y \to \Jet^n(X)$. A detailed description of the structure
of jet bundles is therefore important.

\begin{fact}[\protect{Affine bundle structure, cf.~\cite[Fact~2.8.3]{KKL10}}]\label{fact:affinebundle}
  Let $X$ be a manifold. Then the following facts hold true for any natural
  number $n\geq 0$:
  \begin{enumerate}
  \item The complex space $\Jet^n(X)$ is a manifold. The obvious forgetful
    morphism $\pi_{n,m} : \Jet^n(X) \to \Jet^m(X)$ is holomorphic for all $n
    \geq m \geq 0$. There are natural isomorphisms $\Jet^0(X) \cong X$ and
    $\Jet^1(X) \cong T_X$.
    
  \item\label{it:affinebundle} The morphism $\pi_{n+1,n} : \Jet^{n+1}(X) \to
    \Jet^n(X)$ has the structure of an affine bundle. The associated vector
    bundle of translations on $\Jet^n(X)$ is the pullback of the tangent
    bundle. \qed
  \end{enumerate}
\end{fact}

\begin{notation}\label{not:affinebundle}
  In the setting of Fact~\ref{fact:affinebundle}, if $\sigma, \tau \in
  \Jet^{n+1}(X)$ are two jets whose $n$-th order parts agree,
  $\pi_{n+1,n}(\sigma) = \pi_{n+1,n}(\tau)$, the affine bundle structure
  mentioned in Item~\ref{it:affinebundle} allows to express the difference
  between $\sigma$ and $\tau$ as a tangent vector $\vec v \in T_X|_x$. In this
  context, we write $\vec v = \sigma - \tau$.
\end{notation}

The following descriptions of jets is immediate from the universal property of
the $\Hom$ space.

\begin{fact}[Jet bundles in deformation theory]\label{fact:jetsanddef}
  Let $f : Y \to X$ be a holomorphic map between a compact complex space $Y$
  and a complex manifold $X$. To give an $n$-th order jet at $[f] \in
  \Hom(Y,X)$, it is equivalent to give a holomorphic section $\tau^n : Y \to
  f^* \Jet^n(X) := \Jet^n(X) \times_X Y$. \qed
\end{fact}

\subsection{Time-dependent vector fields}\label{ssec:tdvf} 

We follow the standard approach familiar from the theory of ordinary
differential equations and define a time-dependent vector field on a complex
manifold $X$ as a vector field on the product of $X$ and a ``time axis''.

\newcommand{\td}{\vec A} \newcommand{\sectd}{\vec B}

\subsubsection{Notation} 

The following notation concerning time-dependent vector fields on $X$ and on
the Cartesian product $X\times \bC$ will be used throughout this paper.

\begin{notation}[Cartesian product]\label{not:tvf}
  Let $X$ be a complex manifold. Consider the product $X \times \bC$.  Let $t$
  be the standard coordinate on $\bC$, with associated vector field
  $\frac{\partial}{\partial t} \in H^0 \bigl( X\times\bC,\, T_{X\times\bC}
  \bigr)$. The projection from $X\times \bC$ to the first factor is denoted by
  $p_X:X\times\bC\to X$. Finally, let $j_t : X \to X \times \bC, x \mapsto
  (x,t)$ be the inclusion map.
\end{notation}

\begin{notation}[Time-dependent vector fields]
  Let $X$ be a complex manifold.  A \emph{time-dependent vector field} $\td$
  on $X$ is a vector field $\td$ on the product $X\times\bC$ contained in the
  subspace
  $$
  H^0\bigl( X \times \bC,\, p_X^*(T_X) \bigr) \subset H^0\bigl( X \times
  \bC,\, T_{X \times \bC}\bigr).
  $$
  For any time $t\in\bC$, the restriction of $\td$ to time $t$ is denoted by
  $$
  \td_t := j_t^*(\sigma) \in H^0(X,\,T_X)
  $$
\end{notation}

\begin{defn}[Vector field with constant flow in time]\label{defn:cf}
  If $\td$ is any time-dependent vector field on a complex manifold $X$, set
  $$
  D(\td) := {\textstyle \frac{\partial}{\partial t}} + \td \in H^0 \bigl(
  X\times\bC,\, T_{X\times\bC} \bigr).
  $$
  We call $D(\td)$ the \emph{vector field with constant flow in time
    associated with $\td$}.
\end{defn}

\subsubsection{Calculus involving time-dependent vector fields}

For later reference, we state without proof a formula involving time-dependent
vector fields and their Lie brackets. The formula is easily checked by a
direct computation in local coordinates.

\begin{lem}[Calculus of time-dependent vector fields]\label{lem:tdvfandlie}
  Let $\td, \sectd \in p_X^*(T_X)$ be any two time-dependent vector fields on
  $X$, and let $n \in \bN$ be any number. Using Notation~\ref{not:tvf}, the
  following equations hold.
  \begin{align}
    \label{f1} [ \td, t^n\cdot \sectd ] & = [ t^n\cdot \td, \sectd ] =
    t^n\cdot [\td,\sectd] \\
    \label{f3} \Lie_{D(\td)} \bigl( {\textstyle \frac{t^n}{n!}} \cdot \sectd
    \bigr) & = {\textstyle \frac{t^n}{n!}} \cdot [\td,\sectd] + {\textstyle
      \frac{t^{n-1}}{(n-1)!}}\cdot \sectd + {\textstyle \frac{t^n}{n!}} \cdot
    {\textstyle \frac{\partial}{\partial t}}\sectd
  \end{align}
  \qed
\end{lem}

\subsection{Jets induced by time-dependent vector fields}
\label{ssec:tdinducedjets}

If $X$ is a complex manifold, $x \in X$ a point and $\vec A \in H^0 \bigl(
X\times\bC,\, p_X^*(T_X) \bigr)$ a time-dependent vector field on $X$, then
the local flow of $\vec A$ through $x$ induces a curve germ $\gamma_x$ at $x$,
uniquely determined by the properties $\gamma_x(0)=x$ and $\gamma_x'(t) = \vec
A_t(\gamma_x(t))$ for all $t$. We denote the associated jets as follows.

\begin{defn}[Jets induced by time-dependent vector fields]\label{defn:tdinducedjets}
  Let $\vec A \in H^0 \bigl( X\times\bC,\, p_X^*(T_X) \bigr)$ be a
  time-dependent vector field on a complex manifold $X$. We denote by
  $$
  \tau^n_{\vec A} : X \to \Jet^n(X)
  $$
  the holomorphic section which assigns to each point $x\in X$ the $n$-th
  order jet at $x\in X$ associated with the $\vec A$-integral curve through
  $x$.
\end{defn}

The following is the key observation of this paper and of the previous paper
\cite{KKL10}. In its simplest form, Fact~\ref{fact:tddiffjets} considers two
time-dependent vector fields $\vec A$, $\vec B$ on $X$ whose associated $n$-th
order jets $\tau^n_{\vec A}$ and $\tau^n_{\vec B}$ agree at a point $x \in
X$. It gives a formula for the difference between $(n+1)$-st order jets
$\tau^{n+1}_{\vec A}(x)$ and $\tau^{n+1}_{\vec B}(x)$ which, using the affine
bundle structure of $\Jet^{n+1}(X) \to \Jet^n(X)$, can be identified with a
tangent vector $\tau^{n+1}_{\vec A}(x) - \tau^{n+1}_{\vec B}(x) \in
T_X|_x$. This formula allows for explicit computations of obstruction cocycles
relevant when trying to lift infinitesimal deformations from $n$-th to
$(n+1)$-st order. We refer to \cite[Sect.~3]{KKL10} for a more detailed
discussion.

\begin{fact}[\protect{Difference formula, cf.~\cite[Cor.~2.5.3]{Joerder10}, \cite[Thm.~4.3]{KKL10}}]\label{fact:tddiffjets}
  Let $\vec A_1, \ldots, \vec A_n, \vec B \in H^0 \bigl( X\times\bC,\,
  p_X^*(T_X) \bigr)$ be time-dependent vector fields on a complex manifold
  $X$. Let $x \in X$ be any point such that the induced $n$-th order jets
  agree at $x$,
  $$
  \tau^n_{\vec A_1}(x) = \tau^n_{\vec A_2}(x) = \cdots = \tau^n_{\vec B}(x).
  $$
  Using Fact~\ref{fact:affinebundle} and Notation~\ref{not:affinebundle} to
  identify the difference between the $(n+1)$-st order jets $\tau^{n+1}_{\vec
    A_n}(x)$ and $\tau^{n+1}_{\vec B}(x)$ with a tangent vector in $T_X|_x$,
  the difference is given as
  $$
  \tau^{n+1}_{\vec B}(x) - \tau^{n+1}_{\vec A_n}(x) = \bigl( \Lie_{D(\vec
    A_1)} \circ \cdots \circ \Lie_{D(\vec A_n)} D(\vec B) \bigr)_0(x) \in
  T_X|_x,
  $$
  where $L_{D({\vec A_i})}$ denotes Lie-derivative with respect to the vector
  field $D({\vec A_i})$ with constant flow in time. \qed
\end{fact}

\section{Admissible vector fields}
\label{sec:atdhvf}

We consider the following setup throughout the present section.

\begin{setup}\label{setup:71}
  Let $X$ be a complex manifold, equipped with a Lie-closed subsheaf $\sF
  \subseteq T_X$ of $\bC_X$-modules. Let $Y \subseteq X$ be a reduced complex
  subspace of $X$ with inclusion map $f:Y\to X$, and let further $\sG
  \subseteq \sF_Y$ be an obstruction sheaf for $\sF$, in the sense of
  Definition~\ref{def:obstrShf}.
\end{setup}

The aim of this section is to define and discuss \emph{admissible} vector
fields. These are time-dependent vector fields whose induced jets are
particularly well-behaved when used to deform the inclusion map $f:Y \to
X$. Since the deformation of $f$ will be defined \emph{locally} on $Y$, we do
not assume compactness of $Y$ in this section.

To begin, we fix the notion a time-dependent vector field in $\sF$, see
Definition~\ref{defn:infdefoindbyf}.

\begin{notation}[Time-dependent vector fields in $\sF$]\label{not:tdvfinf}
  A time-dependent vector field $\vec A\in H^0(X\times\bC, p_X^*(T_X))$ is
  said to be a \emph{time-dependent vector field in $\sF$}, if it can be
  expressed as a finite sum
  $$
  \vec A = \sum_{i=1}^nt^i \vec A_i
  $$
  where $\vec A_i\in H^0(X,\sF)$ is a time-independent vector field in $\sF$,
  and $t$ is the standard coordinate on $\bC$.
\end{notation}

\begin{defn}[Admissible vector field]\label{defn:atdhvf}
  In Setup~\ref{setup:71}, let $\vec A \in H^0 \bigl( X\times\bC ,\,
  p_X^*(T_X) \bigr)$ be a time-dependent vector field in $\sF$, and let $n\geq
  1$ be any natural number. The vector field $\vec A$ is said to be
  \emph{$n$-admissible for the obstruction sheaf $\sG$}, if for any natural
  number $1\leq m\leq n$, the restriction to $Y$ of $\Lie_{D(\vec A)}^m\vec A$
  at time $t=0$ is a section of the obstruction sheaf. In other words, if
  $$
  \bigl( \Lie_{D(\vec A)}^m\vec A \bigr)_0\big|_Y \in H^0 \bigl(Y,\, \sG
  \bigr).
  $$
\end{defn}

\begin{rem}[Admissible fields and derivative in time direction]\label{rem:aslj}
  If $\vec A \in H^0 \bigl( X\times\bC ,\, p_X^*(T_X) \bigr)$ is any
  time-dependent vector field on $X$, then
  $$
  \Lie_{D(\vec A)} \vec A = %
  \Lie_{\vec A + \frac{\partial}{\partial t}} \vec A = %
  \Lie_{\frac{\partial}{\partial t}} \vec A = %
  -\Lie_{\vec A} {\textstyle \frac{\partial}{\partial t}} = %
  -\Lie_{D(\vec A)} {\textstyle \frac{\partial}{\partial t}}.
  $$
  More generally, we have $\Lie^m_{D(\vec A)} \vec A = - \Lie^m_{D(\vec A)}
  \frac{\partial}{\partial t}$.  If $\vec A$ is $n$-admissible for $\sG$, then
  $\bigl(\Lie^m_{D(\vec A)} \frac{\partial}{\partial t}\bigr)_0\big|_Y \in H^0
  \bigl( Y,\, \sG \bigr)$ for all $1 \leq m \leq n$.
\end{rem}

\begin{rem}[Time-independent vector fields in $\sF$ are admissible]\label{rem:hfareadmiss}
  A time-independent section of $\sF$ is $n$-admissible for arbitrary $n$ and
  arbitrary obstruction sheaf, when considered as time-dependent vector
  field. In other words, any field
  $$
  \vec A \in H^0\bigl(X,\, \sF\bigr) \subset H^0\bigl(X\times\bC,\,
  p_X^{-1}(\sF)\bigr)
  $$
  is $n$-admissible for any $\sG$.
\end{rem}

\subsection{Jets induced by admissible fields}
\label{ssec:jibatdhvf}

The geometric meaning of Definition~\ref{defn:atdhvf} is perhaps not
obvious. However, the usefulness of the concept will immediately become clear
once we look at jets induced by admissible vector fields.  The following two
propositions, which form the technical core of this paper, summarize the main
features. Proofs are given in Subsections~\ref{ssec:prep}--\ref{sssec:3-21}
below.

\begin{prop}[Extension of jets from $n$-th to $(n+1)$-st order]\label{prop:prop1}
  In Setup~\ref{setup:71}, let $\vec A \in H^0\bigl(X \times \bC,\,p_X^*(T_X)
  \bigr)$ be an $n$-admissible vector field for $\sG$ with $n\geq 1$. If $\vec
  \Delta \in H^0\bigl(X,\,\sF\bigr)\subset H^0\bigl(X ,\, T_X)$ is any
  time-independent vector field in $\sF$ whose restriction to $Y$ lies in the
  obstruction sheaf, $\vec \Delta|_Y\in H^0\bigl(Y,\,\sG\bigr)$, then there
  exists a time-dependent vector field $\vec B \in H^0
  \bigl(X\times\bC,\,p_X^*(T_X) \bigr)$ such that the following holds true.
  \begin{enumerate}
  \item\label{il:3-17-1} The time-dependent vector field $\vec B$ is
    $(n+1)$-admissible for $\sG$.
  \item\label{il:3-17-2} The $n$-th order jets induced by $\vec A$ and $\vec
    B$ agree on $Y$,
    $$
    \tau^n_{\vec A}\big|_{Y} = \tau^n_{\vec B}\big|_{Y}.
    $$
  \item\label{il:3-17-3} The difference between the induced $(n+1)$-st order
    jets is given by $\vec \Delta$,
    $$
    \tau^n_{\vec B}\big|_{Y}-\tau^n_{\vec A}\big|_{Y} = \vec \Delta\big|_{Y}.
    $$
  \end{enumerate}
\end{prop}

\begin{prop}[Differences of jets induced by admissible vector fields]\label{prop:prop2}
  In Setup~\ref{setup:71}, let $n \geq 1$ be any number and $\vec A, \vec B
  \in H^0 \bigl(X\times\bC,\, p_X^*(T_X) \bigr)$ be two time-dependent vector
  fields, both of them $n$-admissible for the obstruction sheaf $\sG$. If the
  induced $n$-th order jets agree on $Y$,
  $$
  \tau^n_{\vec A}\big|_{Y} = \tau^n_{\vec B}\big|_{Y},
  $$
  then the difference between the $(n+1)$-st order deformations of $f$ lies in
  the obstruction sheaf $\sG$. In other words,
  $$
  \tau^{n+1}_{\vec B}\big|_{Y} - \tau^{n+1}_{\vec A}\big|_{Y} \in H^0 \bigl(
  Y,\, \sG \bigr).
  $$
\end{prop}

The proofs of Propositions~\ref{prop:prop1} and \ref{prop:prop2} are quite
elementary, but somewhat lengthy and tedious. The reader interested in gaining
an overview of the argumentation is advised to skip
Sections~\ref{ssec:prep}--\ref{sssec:3-21} on first reading and continue with
Section~\ref{ssec:higherorderadm} on page~\pageref{ssec:higherorderadm}, where
Propositions~\ref{prop:prop1} and \ref{prop:prop2} are used to lift
first-order infinitesimal deformations of $f$ to arbitrary order.

\subsection{Preparation for the proof of Proposition~\ref*{prop:prop1}}
\label{ssec:prep}

The proof of Proposition~\ref{prop:prop1} relies on the following
computational lemma\PreprintAndPublication{, which we formulate and prove in
  the current preparatory Section~\ref{ssec:prep}}{}.

\begin{lem}\label{lem:prop1}
  In the setup of Proposition~\ref{prop:prop1}, let $\vec E$ be any
  time-independent vector field in $\sF$ and consider the time-dependent
  vector field in $\sF$
  \begin{equation}\label{eq:defB}
    \vec B := \vec A + \frac{t^n}{n!} \vec \Delta + \frac{t^{n+1}}{(n+1)!}\vec E.    
  \end{equation}
  Then the following equalities hold up to higher-order terms of $t$, for all
  $1 \leq m \leq n$,
  \begin{align}
    \label{eq:camillo} \Lie_{D(\vec B)}^m (\vec B) & \equiv \Lie_{D(\vec
      A)}^m(\vec A) +
    {\textstyle\frac{t^{n-m}}{(n-m)!}}\vec \Delta && \mod (t^{n-m+1}) \\
    \label{eq:peppone} \Lie_{D(\vec B)}^{n+1}\vec B & \equiv
    \Lie_{D(\vec A)}^{n+1}(\vec A)+ \vec E + n[\vec A,\vec
    \Delta] && \mod (t).
  \end{align}
\end{lem}

\PreprintAndPublication{We will prove Lemma~\ref{lem:prop1} in the
  remainder of the present Section~\ref{ssec:prep}.

  \subsubsection*{Proof of Lemma~\ref{lem:prop1}}

  If $\vec F \in H^0 \bigl(X\times\bC,\,p_X^*(T_X)\bigr)$ is any
  time-dependent vector field on $X$, Equation~\eqref{f1} of
  Lemma~\ref{lem:tdvfandlie} shows that
  \begin{equation}\label{eq:ind1}
    \Lie_{D(\vec B)}{\vec F} = \Lie_{D(\vec A)}{\vec F} + {\textstyle \frac{t^n}{n!}} [\vec \Delta, {\vec F}] +
    {\textstyle \frac{t^{n+1}}{(n+1)!}} [\vec E, {\vec F}].
  \end{equation}
  We use equation~\eqref{eq:ind1} to prove Lemma~\ref{lem:prop1} by induction.

  \subsubsection*{Proof of Equation~\eqref{eq:camillo}}

  Let $1\leq m\leq n$ be any given number. Rather than proving
  Equation~\eqref{eq:camillo} directly, we will show by induction that the
  following stronger statement holds true modulo $(t^{n-m+2})$
  \begin{equation}\label{eq:crj}
    \Lie_{D(\vec B)}^m (\vec B)  \equiv \Lie_{D(\vec A)}^m(\vec A) + {\textstyle
      \frac{t^{n-m}}{(n-m)!}}\vec \Delta + {\textstyle \frac{t^{n-m+1}}{(n-m+1)!}}
    \bigl(\vec E + (m-1)[\vec A,\vec \Delta]\bigr)
  \end{equation}
  To start the inductive proof of Equation~\eqref{eq:crj}, we compute the Lie
  derivative $\Lie_{D(\vec B)}(\vec B)$ modulo $(t^{n+1})$,
  \begin{align*}
    \Lie_{D(\vec B)}(\vec B) & \equiv \Lie_{D(\vec A)}(\vec B) +
    {\textstyle \frac{t^n}{n!}}[\vec \Delta,\vec B]
    && \text{by \eqref{eq:ind1}} \\
    & \equiv \Lie_{D(\vec A)}(\vec B) + {\textstyle
      \frac{t^n}{n!}}[\vec \Delta,\vec A] && \text{by
      \eqref{f1}} \\
    & \equiv \Lie_{D(\vec A)}(\vec A) + \Lie_{D(\vec A)} \left(
      {\textstyle \frac{t^n}{n!}} \vec \Delta \right) + \Lie_{D(\vec
      A)} \left( {\textstyle \frac{t^{n+1}}{(n+1)!}}\vec E \right) +
    {\textstyle \frac{t^n}{n!}}[\vec \Delta,\vec A] && \text{by \eqref{eq:defB}} \\
    & \equiv \Lie_{D(\vec A)}(\vec A) + {\textstyle
      \frac{t^{n-1}}{(n-1)!}}\vec \Delta + {\textstyle
      \frac{t^n}{n!}}[\vec A,\vec \Delta] + {\textstyle
      \frac{t^n}{n!}} \vec E +
    {\textstyle \frac{t^n}{n!}}[\vec \Delta,\vec A] &&  \text{by \eqref{f3}} \\
    & \equiv \Lie_{D(\vec A)}(\vec A) + {\textstyle
      \frac{t^{n-1}}{(n-1)!}}\vec \Delta + {\textstyle \frac{t^n}{n!}}
    \vec E
  \end{align*}
  This proves the Equation~\eqref{eq:crj} in case $m=1$.

  For the inductive step, suppose that Equation~\eqref{eq:crj} holds true
  modulo $(t^{n-m+2})$ for some given number $1\leq m< n$. That is, assume
  that there exists a time-dependent vector field ${\vec F} \in H^0\bigl(
  X\times\bC ,\, p_X^*(T_X) \bigr)$ such that the following holds
  \begin{multline}\label{eq:xcf}
    \Lie_{D(\vec B)}^m(\vec B) = \Lie_{D(\vec A)}^m(\vec A) +
    {\textstyle \frac{t^{n-m}}{(n-m)!}}\vec \Delta + {\textstyle
      \frac{t^{n-m+1}}{(n-m+1)!}}  \bigl(\vec E + (m-1)[\vec A,\vec
    \Delta] \bigr) + \\ t^{n-m+2}{\vec F}.
  \end{multline}
  We obtain the following sequence of equalities modulo $(t^{n-m+1})$.
  \begin{align*}
    & \Lie_{D(\vec B)}^{m+1}(\vec B) \equiv \Lie_{D(\vec B)}
    \Lie_{D(\vec B)}^m(\vec B) \equiv
    \Lie_{D(\vec A)} \Lie_{D(\vec B)}^m(\vec B) && \text{by \eqref{eq:ind1}} \\
    & \equiv\Lie_{D(\vec A)}\Bigl(\Lie_{D(\vec A)}^m(\vec A) +
    {\textstyle \frac{t^{n-m}}{(n-m)!}}\vec \Delta + {\textstyle
      \frac{t^{n-m+1}}{(n-m+1)!}}  \bigl(\vec E +
    (m-1)[\vec A,\vec \Delta] \bigr) \Bigr) && \text{by \eqref{eq:xcf}} \\
    & \equiv \Lie_{D(\vec A)}^{m+1}(\vec A) +
    {\textstyle\frac{t^{n-(m+1)}}{(n-(m+1))!}}\vec \Delta +
    {\textstyle\frac{t^{n-m}}{(n-m)!}}[\vec A,\vec \Delta]
    +  {\textstyle \frac{t^{n-m}}{(n-m)!}} \bigl(\vec E + (m-1)[\vec A,\vec \Delta] \bigr)  &&  \text{by \eqref{f3}}\\
    &\equiv \Lie_{D(\vec A)}^{m+1}(\vec
    A)+{\textstyle\frac{t^{n-(m+1)}}{(n-(m+1))!}}\vec \Delta +
    {\textstyle \frac{t^{n-m}}{(n-m)!}} \bigl(\vec E+m[\vec A,\vec
    \Delta] \bigr)
  \end{align*}
  Equation~\eqref{eq:crj} thus holds true for all $1\leq m\leq n$.  This
  proves \eqref{eq:camillo} for all numbers $1 \leq m \leq n$.

  \subsubsection*{Proof of Equation~\eqref{eq:peppone} and end of proof}

  Equation~\eqref{eq:peppone} follows from the same computation as before,
  considering Equation~\eqref{eq:crj} in case $m = n$, and applying
  Equation~\eqref{eq:ind1} to compute $\Lie_{D(\vec B)}^{n+1}(\vec B) =
  \Lie_{D(\vec B)} \Lie_{D(\vec B)}^n(\vec B)$ modulo $(t)$. This finishes the
  proof of Lemma~\ref{lem:prop1}. \qed

}{Lemma~\ref{lem:prop1} easily follows from a direct but rather tedious
  computation. Details are found in the preprint version of this paper,
  available on the arXiv.}

\subsection{Proof of Proposition~\ref*{prop:prop1}}

Consider the time-dependent vector field
\begin{equation}\label{eq:dB2}
  \vec B := \vec A + \frac{t^n}{n!}\vec \Delta + \frac{t^{n+1}}{(n+1)!}\vec E,
  \quad\text{where}\quad \vec E := -n [\vec A,\vec \Delta] - \bigl(\Lie_{D(\vec A)}^{n+1}\vec A\bigr)_0.
\end{equation}

Since $\vec \Delta|_Y$ is a section of the obstruction sheaf $\sG$,
Equation~\eqref{eq:camillo} immediately implies that $\vec B$ is
$n$-admissible. By choice of $\vec E$, Equation~\eqref{eq:peppone} shows that
$\vec B$ is in fact $(n+1)$-admissible. This already shows
Property~(\ref{il:3-17-1}) claimed in Proposition~\ref{prop:prop1}.

\subsubsection*{Proof of Proposition~\ref*{prop:prop1}, Property~(\ref{il:3-17-2})}

To show Property~(\ref{il:3-17-2}), we will prove by induction that
\begin{equation}\label{eq:jagr}
  \tau^m_{\vec B}\big|_{Y} = \tau^m_{\vec A}\big|_{Y}, \quad\text{for all } 1 \leq m \leq n. 
\end{equation}

We start the induction with the case where $m = 1$. In this case
Equation~\eqref{eq:jagr} simply asserts that the vector fields $D(\vec A)$ and
$D(\vec B)$ agree along $Y$ at time $t=0$. That, however, is clear by choice
of $\vec B$.

For the inductive step, assume that Equation~\eqref{eq:jagr} was shown for a
certain number $1 \leq m < n$.  It will then follow from
Fact~\ref{fact:tddiffjets} that the difference between the $(m+1)$-st order
jets is given as
\begin{equation}\label{eq:dj}
  \tau^{m+1}_{\vec B}\big|_{Y} - \tau^{m+1}_{\vec A}\big|_{Y} =
  \bigl( \Lie_{D(\vec A)}^m D(\vec B) \bigr)_0 \Bigr|_{Y}.
\end{equation}
But since
\begin{align*}
  \Lie_{D(\vec A)}^m D(\vec B) &= \Lie_{D(\vec A)}^m \left( D(\vec A)+
    {\textstyle \frac{t^n}{n!}} \vec \Delta + {\textstyle
      \frac{t^{n+1}}{(n+1)!}} \vec E \right) && \text{by~\eqref{eq:dB2}}\\
  &= \Lie_{D(\vec A)}^m \left( {\textstyle \frac{t^n}{n!}} \vec \Delta +
    {\textstyle \frac{t^{n+1}}{(n+1)!}} \vec E \right)\\
  &\equiv {\textstyle \frac{t^{n-m}}{(n-m)!}}\vec \Delta \mod (t^{n-m+1}) &&
  \text{by~\eqref{f3}},
\end{align*}
it is clear that the difference~\eqref{eq:dj} vanishes as required.  This
finishes the proof of Property~(\ref{il:3-17-2}) claimed in
Proposition~\ref{prop:prop1}.

\subsubsection*{Proof of Proposition~\ref*{prop:prop1}, Property~(\ref{il:3-17-3})}

Fact~\ref{fact:tddiffjets} and Property~(\ref{il:3-17-2}) together imply that
the difference between the $(n+1)$-st order jets is given by
$$
\tau^{n+1}_{\vec B}\big|_{Y} - \tau^{n+1}_{\vec A}\big|_{Y} = \bigl(
\Lie_{D(\vec A)}^nD(\vec B)\bigr)_0 \Bigr|_{Y}.
$$
As in the proof of Property~(\ref{il:3-17-2}) above we obtain that
$$
\Lie_{D(\vec A)}^n D(\vec B) \equiv \vec \Delta \mod (t),
$$
finishing the proof of Proposition~\ref{prop:prop1}. \qed

\subsection{Preparation for the proof of Proposition~\ref*{prop:prop2}}\label{ssec:prepprop2}

The proof of Proposition~\ref{prop:prop2} makes use of two computational
lemmas concerning Lie derivatives that are formulated and proved in the
current Section~\ref{ssec:prepprop2}. The actual proof of
Proposition~\ref{prop:prop2} is given in Section~\ref{sssec:3-21}.  To start,
recall the classical Jacobi identity, formulated in terms of Lie derivatives.

\begin{rem}[Jacobi identity]\label{rem:jacobi}
  Let $\vec X, \vec Y, \vec Z \in H^0(X,T_M)$ be vector fields on a complex
  manifold $M$. Written in terms of Lie derivatives rather than Lie brackets,
  the Jacobi identity asserts that
  \begin{equation}\label{eq:jacobi}
    \Lie_{[\vec X,\vec Y]}\vec Z = \Lie_{\vec X} \circ \Lie_{\vec Y} \vec
    Z - \Lie_{\vec Y} \circ \Lie_{\vec X}\vec Z.
  \end{equation}
\end{rem}

\begin{lem}\label{lem:hering}
  In the setup of Proposition~\ref{prop:prop2}, define
  \begin{align*}
    \vec R^2 & := [D(\vec A), D(\vec B)] = \Lie_{D(\vec A)} D(\vec B)  && \text{and inductively}\\
    \vec R^m & := [D(\vec A), \vec R^{m-1}] = \Lie_{D(\vec A)}^{m-1} D(\vec B)
    && \text{for $m>2$.}
  \end{align*}
  If $\vec Z\in H^0\bigl(X\times \bC,\, T_{X\times\bC}\bigr)$ is any vector
  field on the product $X\times\bC$ and $m\geq 2$ any number, then $\Lie_{\vec
    R^m} \vec Z$ can be expressed as a linear combination of terms $\vec T$ of
  the form
  $$
  \vec T = \Lie_{\vec F_1}\circ \cdots \circ \Lie_{\vec F_m} \bigl( \vec Z
  \bigr),
  $$
  where all $\vec F_i$ are either equal to $D(\vec A)$ or equal to $D(\vec
  B)$.
\end{lem}
\begin{proof}
  We prove Lemma~\ref{lem:hering} by induction on $m$. If $m = 2$, then the
  Jacobi Identity~\eqref{eq:jacobi} asserts that
  $$
  \Lie_{\vec R^2} \vec Z = \Lie_{[D(\vec A), D(\vec B)]} \vec Z = \Lie_{D(\vec
    A)} \circ \Lie_{D(\vec B)} \vec Z - \Lie_{D(\vec B)} \circ \Lie_{D(\vec
    A)} \vec Z,
  $$
  showing the claim in case where $m=2$. For $m > 2$, the Jacobi identity
  gives
  $$
  \Lie_{\vec R^m} \vec Z = \Lie_{[D(\vec A), \vec R^{m-1}]} \vec Z =
  \Lie_{D(\vec A)} \circ \Lie_{\vec R^{m-1}} \vec Z - \Lie_{\vec R^{m-1}}\circ
  \Lie_{D(\vec A)} \vec Z,
  $$
  showing the claim inductively.
\end{proof}

\begin{lem}\label{lem:prop2}
  In the setup of Proposition~\ref{prop:prop2}, if $\vec F_1, \ldots, \vec
  F_n$ is a sequence of time-dependent vector fields such that all $\vec F_i$
  are either equal to $D(\vec A)$ or equal to $D(\vec B)$, then
  \begin{equation}\label{eq:lp2}
    \Bigl(\Lie_{\vec F_1}\circ \cdots \circ \Lie_{\vec F_n} \bigl( {\textstyle
      \frac{\partial}{\partial t}} \bigr) \Bigr)_0 \Bigl|_Y \in H^0 \bigl( Y,\,
    \sG \bigr).
  \end{equation}
\end{lem}
\begin{proof}
  If all $\vec F_i$ are equal to $D(\vec A)$ or all $\vec F_i$ are equal to
  $D(\vec B)$, then the statement follows from Remark~\ref{rem:aslj}. We can
  thus assume without loss of generality that $\vec F_n = D(\vec B)$, and that
  at least one of the $\vec F_i$, for $i<n$, is equal to $D(\vec A)$.

  As a first step in the proof of Lemma~\ref{lem:prop2}, we show the following
  claim.

  \begin{subclaim}\label{sclaim:1}
    Let $1 \leq i \leq n-2$ be any number, and consider the vector
    field\footnote{The equality in~\eqref{eq:xxy} is again Jacobi's identity
      in the form of Remark~\ref{rem:jacobi}.}
    \begin{equation}\label{eq:xxy}
      \begin{split}
        \vec T & := \Lie_{\vec F_1} \circ \cdots \circ \Lie_{\vec F_{i-1}}
        \circ \Lie_{[\vec F_i,\vec F_{i+1}]} \circ \Lie_{\vec F_{i+2}} \circ
        \cdots \circ \Lie_{\vec F_n}\bigl({\textstyle \frac{\partial}{\partial
            t}}
        \bigr) \\
        & \,\,= \Lie_{\vec F_1} \circ \cdots \circ \Lie_{\vec F_{i-1}} \circ
        \Lie_{\vec F_i} \circ \Lie_{\vec F_{i+1}} \circ \Lie_{\vec F_{i+2}}
        \circ \cdots \circ \Lie_{\vec F_n}\bigl({\textstyle
          \frac{\partial}{\partial t}}
        \bigr) - \\
        & \qquad \Lie_{\vec F_1} \circ \cdots \circ \Lie_{\vec F_{i-1}} \circ
        \Lie_{\vec F_{i+1}} \circ \Lie_{\vec F_i} \circ \Lie_{\vec F_{i+2}}
        \circ \cdots \circ \Lie_{\vec F_n}\bigl({\textstyle
          \frac{\partial}{\partial t}} \bigr).
    \end{split}
    \end{equation}
    Then $\vec T_0\big|_Y \in H^0 \bigl( Y,\, \sG \bigr)$.
  \end{subclaim}

  Using a somewhat tedious inductive argument, which we leave to the reader,
  one verifies that the vector field $\vec T$ can be expressed as follows,
  $$
  \vec T := \sum_{\txt{\tiny $\alpha$ a subsequence\\\tiny of $( 1, 2, \ldots,
      i-1 )$}}[\vec P^\alpha,\,\vec T^\alpha]
  $$
  where the subsequence $\alpha$ is written as, $\alpha = \bigl( \alpha(1),
  \ldots, \alpha(k) \bigr)$, where $\overline \alpha = \bigl( \overline
  \alpha(1), \ldots, \overline \alpha(i-1-k) \bigr)$ denotes the complementary
  subsequence, and where
  \begin{align*}
    \vec P^\alpha & := \Lie_{\vec F_{\alpha(1)}}\circ\cdots\circ\Lie_{\vec
      F_{\alpha(k)}}\circ\Lie_{\vec F_i}\bigl(\vec F_{i+1}\bigr) & \text{and}\\
    \vec T^\alpha & := \Lie_{\vec F_{\overline\alpha(1)}} \circ \cdots \circ
    \Lie_{\vec F_{\overline\alpha(l)}} \circ \Lie_{\vec F_{i+2}} \circ \cdots
    \circ \Lie_{\vec F_n}\bigl({\textstyle \frac{\partial}{\partial t}} \bigr).
  \end{align*}
  There are two things to note in this setting.
  \begin{enumerate}
  \item Since $\sF$ is closed under Lie-bracket, the restriction to time $t=0$
    of the vector field $\vec T^\alpha$ is a section of $\sF$, that is,
    $\bigl( \vec T^\alpha \bigr)_0 \in H^0 \bigl( X,\, \sF \bigr)$, for all
    subsequences $\alpha$.
  \item The vector field $(\vec P^{\alpha})_0 \in H^0 \bigl(X, \, \sF \bigr)$
    is the restriction to time $t=0$ of an iterated Lie bracket of at most
    $(n-1)$ time-dependent vector fields in $\sF$, all inducing the same
    $(n-1)$-st order jets on $Y \subseteq X$.  Fact~\ref{fact:tddiffjets}
    therefore implies that $(\vec P^\alpha)_0$ vanishes on $Y$, again for all
    subsequences $\alpha$.
  \end{enumerate}
  Consequently, the restriction of $\vec T_0$ to $Y$ satisfies
  $$
  \vec T_0|_Y = \sum_\alpha \,\, [\vec P^\alpha_0, \vec T^\alpha_0] =
  \sum_\alpha \,\, [\vec P^\alpha_0+\vec T^\alpha_0, \vec T^\alpha_0] \in
  H^0\bigl( Y,\, \sG \bigr),
  $$
  by Definition~\ref{def:obstrShf}.  This ends the proof of
  Claim~\ref{sclaim:1}.
  
  \medskip
  
  Claim~\ref{sclaim:1} asserts that Equation~\eqref{eq:lp2} holds if and only
  if it holds after permuting the operators $ \Lie_{\vec F_i}$ and $\Lie_{\vec
    F_{i+1}}$. In other words, Equation~\eqref{eq:lp2} holds if and only if
  $$
  \Bigl(\Lie_{\vec F_1}\circ \cdots \circ \Lie_{\vec F_{i+1}} \circ \Lie_{\vec
    F_i} \circ \cdots \circ \Lie_{\vec F_n} \bigl( {\textstyle
    \frac{\partial}{\partial t}} \bigr) \Bigr)_0 \Bigl|_Y \in H^0 \bigl( Y,\,
  \sG \bigr).
  $$
  Using the classical ``bubble-sort'' algorithm, we can therefore sort the
  $\vec F_i$ and assume without loss of generality that there exists a number
  $k$ such that $\vec F_1, \vec F_2, \ldots, \vec F_k$ are all equal to
  $D(\vec A)$, whereas $\vec F_{k+1}, \ldots, \vec F_n$ are all equal to
  $D(\vec B)$.  In this situation an argument similar to, but easier than the
  proof of Claim~\ref{sclaim:2} then shows the following.

  \begin{subclaim}\label{sclaim:2}
    If $\vec S$ denotes the following vector field,
    $$
    \vec S := \Lie_{\vec F_1} \circ \cdots \circ \Lie_{\vec
      F_n}\bigl({\textstyle \frac{\partial}{\partial t}} \bigr) - \Lie_{D(\vec
      B)} \circ \Lie_{\vec F_2}\circ \cdots \circ \Lie_{\vec
      F_n}\bigl({\textstyle \frac{\partial}{\partial t}} \bigr),
    $$
    then $\vec S_0|_Y \in H^0 \bigl( Y,\, \sG \bigr)$. 
  \end{subclaim}
  
  The proof of Claim~\ref{sclaim:2} is left to the reader.
  Claim~\ref{sclaim:2} allows to replace $\vec F_1 = D(\vec A)$ by $D(\vec
  B)$. Applying the 'Sorting Claim~\ref{sclaim:1}' and the 'Replacement
  Claim~\ref{sclaim:2}' exactly $k$ times, this reduces the
  Lemma~\ref{lem:prop2} to the case where $\vec F_1 = \cdots = \vec F_n =
  D(\vec B)$, where the lemma is known to be true, thus finishing the proof of
  Lemma~\ref{lem:prop2}.
\end{proof}

\subsection{Proof of Proposition~\ref*{prop:prop2}}
\label{sssec:3-21}

By Fact~\ref{fact:tddiffjets}, the time-independent vector field
$$
\vec \Delta := \bigl( \Lie_{D(\vec A)}^n D(\vec B) \bigr)_{0} \in H^0
\bigl(X,\,T_X \bigr)
$$
at time $t=0$ describes the difference between the $(n+1)$-st order jets
$\tau^{n+1}_{\vec A}$ and $\tau^{n+1}_{\vec B}$ along $Y \subseteq X$. We need
to show that the restriction $\vec \Delta|_Y$ is a section of the obstruction
sheaf $\sG$. For simplicity of argument, we discuss the cases $n=1$ and $n>1$
separately in the next two subsections.

\subsubsection*{Proof in case $n = 1$}

If $n=1$, expand the definition of $\vec \Delta$ and of $D(\vec A)$ and
$D(\vec B)$ to obtain
$$
\vec \Delta\big|_Y = \bigl(\Lie_{D(\vec A)} D(\vec B)\bigr)_0\big|_Y =
\bigl(\Lie_{\vec A} \vec B\bigr)_0\big|_Y +
\underbrace{\bigl(\Lie_{\frac{\partial}{\partial t}} \vec
  B\bigr)_0\big|_Y}_{\in \sG \text{ by Rem.~\ref{rem:aslj}}} +
\underbrace{\bigl(\Lie_{\vec A} {\textstyle \frac{\partial}{\partial
      t}}\bigr)_0\big|_Y}_{\in \sG \text{ by Rem.~\ref{rem:aslj}}}.
$$
To conclude, it thus suffices to show that
$$
\bigl(\Lie_{\vec A} \vec B\bigr)_0\big|_Y = \bigl(\Lie_{\vec A_0} \vec B_0
\bigr)\big|_Y \in H^0 \bigl( Y,\, \sG \bigr).
$$
This, however, follows from the assumption that $\vec A$ and $\vec B$ be
admissible vector fields, so that $\vec A_0, \vec B_0 \in H^0 \bigl(X,\, \sF
\bigr)$, and from Definition~\ref{def:obstrShf} of \emph{obstruction
  sheaf}. This proves Proposition~\ref{prop:prop2} in case where $n=1$. \qed

\subsubsection*{Proof in case $n > 1$}

Consider the time-dependent vector field
$$
\vec R := \Lie_{D(\vec A)}^{n-1} D(\vec B).
$$
Fact~\ref{fact:tddiffjets} and the assumed equality of $n$-th order jets imply
that $\vec R_0|_{Y} = 0$. As one consequence, we obtain that
$$
\bigl( L_{\vec R_0}\vec A_0 \bigr)\big|_Y = \bigl( L_{\vec R_0+\vec A_0}
\vec A_0 \bigr)\big|_Y \in H^0\bigl( Y,\, \sG \bigr),
$$
according to Definition~\ref{def:obstrShf}, where obstruction sheaves were
introduced. The equation
$$
\vec \Delta\big|_Y = \bigl( \Lie_{D(\vec A)}\vec R \bigr)_{0}\big|_Y =
- \bigl( \Lie_{\vec R}D(\vec A) \bigr)_{0} \big|_Y =
\underbrace{-\Lie_{\vec R_0} \bigl( \vec A_0 \bigr)\big|_Y}_{\in \sG}
- \bigl( \Lie_{\vec R} {\textstyle \frac{\partial}{\partial t}}
\bigr)_0\big|_Y
$$
thus asserts that to prove Proposition~\ref{prop:prop2}, it suffices to show
that
\begin{equation}\label{eq:cprp}
  \bigl(\Lie_{\vec R} {\textstyle \frac{\partial}{\partial t}}
  \bigr)_0 \big|_Y\in H^0 \bigl( Y, \, \sG \bigr).
\end{equation}
To this end, recall from Lemma~\ref{lem:hering} that $\Lie_{\vec R}
\frac{\partial}{\partial t}$ can be expressed as a linear combination of terms
$\vec T$ of the form
$$
\vec T = \Lie_{\vec F_1}\circ \cdots \circ \Lie_{\vec F_n} \bigl(
{\textstyle \frac{\partial}{\partial t}} \bigr)
$$
with $\vec F_i\in\{D(\vec A),D(\vec B)\}$ for $1\leq i\leq n$.
Lemma~\ref{lem:prop2} then asserts that every one of these terms is contained
in $H^0\bigl( Y,\, \sG \bigr)$ when restricted at time $t=0$ to $Y$, thus
finishing the proof of Proposition~\ref{prop:prop2} in case where $n >
1$. \qed

\section{Admissible higher-order infinitesimal deformations}
\label{ssec:higherorderadm}

In this section, we generalize the notion of admissibility to jets of
arbitrary order. We employ Proposition~\ref{prop:prop1} and~\ref{prop:prop2}
to show that - under a suitable cohomology vanishing assumption - first-order
infinitesimal deformations that are locally induced by admissible vector
fields can always be lifted to \emph{admissible} infinitesimal deformations of
arbitrary order. We maintain the assumptions spelled out in
Setup~\ref{setup:71} on page \pageref{setup:71}.

\begin{defn}[Admissible higher-order infinitesimal deformations]\label{defn:admissiblesection}
  In Setup~\ref{setup:71}, let $n\geq 1$. A section $\tau^n:Y\to \Jet^n(X)$
  over $f$ is said to be \emph{admissible for the obstruction sheaf $\sG$}, if
  there exists a cover of $Y \subset X$ by open subsets, say $Y \subseteq
  \bigcup_{j \in J} X_j^\circ$, and time-dependent vector fields $\vec A_j$ on
  $X_j^\circ$ for any $j \in J$ such that the following two conditions hold
  true for every $j\in J$.
  \begin{enumerate}
  \item\label{antje} The vector field $\vec A_j$ is $n$-admissible for the
    obstruction sheaf $\sG|_{Y_i^\circ}$, where $Y_j^\circ :=Y \cap
    X_j^\circ$.
  \item\label{marina} The restriction of $\tau^n$ to $Y_j^\circ$ is induced by
    $\vec A_j$. In other words, $\tau^n\big|_{Y_j^\circ}=\tau^n_{\vec
      A_j}\big|_{Y_j^\circ}$.
  \end{enumerate}
\end{defn}

\begin{rem}
  Recall from Remark~\ref{rem:hfareadmiss} that time-independent vector fields
  in $\sF$ are always admissible.  In Setup~\ref{setup:71}, any infinitesimal
  deformation $\sigma \in H^0\bigl( Y,\, \sF_Y \bigr)$ locally induced by
  $\sF$ is therefore admissible for the obstruction sheaf $\sG$ in the sense
  of Definition~\ref{defn:admissiblesection} above.
\end{rem}

\begin{prop}[Lifting admissible sections to arbitrary order]\label{prop:liftingproperty}
  In Setup~\ref{setup:71}, suppose that $H^1\bigl(Y,\,\sG\bigr)=0$. Then any
  section $\tau^n:Y\to \Jet^n(X)$ over $f$ that is admissible for $\sG$ admits
  a lift to a section $\tau^{n+1}:Y\to \Jet^{n+1}(X)$ that is likewise
  admissible for $\sG$.
\end{prop}
\begin{proof}
  Fix an open cover $Y \subseteq \bigcup_{j\in J} X_j^\circ$ and
  $n$-admissible vector fields $\vec A_j$ as in
  Definition~\ref{defn:admissiblesection}. We write $Y_j^\circ := X_j^\circ
  \cap Y$, $X_{jk}^\circ := X_j^\circ \cap X_k^\circ$ and $Y_{jk}^\circ :=
  Y_j^\circ \cap Y_k^\circ$ for $j,k\in J$.
  
  Since $\tau^{n}_{\vec A_j}\big|_{Y_{jk}^\circ} = \tau^{n}_{\vec
    A_k}\big|_{Y_{jk}^\circ}$ for $j,k\in J$ by Item~(\ref{marina}) in
  Assumptions~\ref{defn:admissiblesection}, we may calculate the difference
  between $(n+1)$-st order jets. By Assumption~(\ref{antje}) and by
  Proposition~\ref{prop:prop2}, this difference is described by a section in
  the obstruction sheaf $\sG$, that is,
  \begin{equation}\label{eq:cocycle}
    \vec C_{jk}:=\tau^{n+1}_{\vec A_j}\big|_{Y_{jk}^\circ} -
    \tau^{n+1}_{\vec A_k}\big|_{Y_{jk}^\circ} \in \sG(Y_{jk}^\circ)
  \end{equation}
  The general properties of affine bundles imply that the family $(\vec
  C_{jk})_{jk}$ is a \v{C}ech $1$-cocycle of the sheaf $\sG$ with respect to
  the open cover $Y = \bigcup_j Y_j^\circ$. The cohomology vanishing
  assumption ensures that $(\vec C_{jk})_{jk}$ is a \v{C}ech $1$-coboundary,
  that is, that there exist sections $\vec C_j\in \sG(Y_j^\circ)$ such that
  \begin{equation}\label{eq:coboundary}
    \vec C_{jk} = \vec C_k\big|_{Y_{jk}^\circ} - \vec C_j\big|_{Y_{jk}^\circ} 
    \in \sG(Y_{jk}^\circ).
  \end{equation}
  Taking~(\ref{eq:cocycle}) and~(\ref{eq:coboundary}) together, the sections
  \begin{equation}\label{eq:lift}
    \tau^{n+1}_{\vec A_j}\big|_{Y_j^\circ} + \vec C_j: Y_i^\circ \to \Jet^{n+1}(X),
  \end{equation}
  with $j$ running over the index set $J$, glue, giving a section $\tau^{n+1}
  : Y \to \Jet^{n+1}(X)$ over $\tau^n$.

  Refining the open cover, Equation~(\ref{eq:lift}) and
  Proposition~\ref{prop:prop1} allow to assume that $\vec C_j$ is induced by a
  vector field on $X_j^\circ$ which is $(n+1)$-admissible for
  $\sG|_{Y_j^\circ}$, as required.
\end{proof}

\section{Proof of Theorems~\ref*{thm:main} and \ref*{thm:mainaddon}}
\label{sec:mainproof}

We maintain assumptions and notation of Theorems~\ref{thm:main} and
\ref{thm:mainaddon}, where $X$ is a complex manifold, $\sF \subset T_X$ a
Lie-closed subsheaf of $\bC_X$-modules, where $Y\subset X$ is a reduced,
compact complex subspace with inclusion map $f : Y \to X$, and $\sG$ an
obstruction sheaf for $\sF$.  Let $\sigma\in H^0\bigl(Y,\,\sF_Y\bigr)$ be an
infinitesimal deformation of $f$ that is locally induced by $\sF$, and assume
that $H^1 \bigl( Y,\, \sG \bigr) = 0$.

\subsection{Proof of the Theorem~\ref*{thm:main}}
\label{ssec:proofmain}

By Proposition~\ref{prop:liftingproperty}, the section $\tau^1:Y\to \Jet^1(X)$
corresponding to the first order deformation $\sigma\in
H^0\bigl(Y,\,\sF_Y\bigr)$ of $f$ inductively admits lifts to sections
$$
\tau^n : Y \to \Jet^n(X)
$$
for any natural number $n$. The family $(\tau^n)_{n \in \bN}$ corresponds to a
formal curve
\begin{equation}\label{eq:formal}
  \Spec\,\bC[[t]] \to \Hom(Y,X)  
\end{equation}
at $[f]\in \Hom(Y,X)$ whose associated Zariski tangent vector equals
$\sigma$. The existence of a holomorphic curve at $[f]\in \Hom(Y,X)$ with
derivative $\sigma$ then follows from a classical result of Michael Artin,
\cite[Thm.~1.2]{ARTIN68}. \qed

\subsection{Proof of Theorem~\ref*{thm:mainaddon}}

If $H \subseteq \Hom(Y,X)$ is a subspace of deformations along $\sF$, then the
formal curve~\eqref{eq:formal} factors through $H$,
$$
\Spec\, \bC[[t]] \to H \subseteq \Hom(Y,X).
$$
In particular, the holomorphic curve at $[f] \in \Hom(Y,X)$ given by
\cite{ARTIN68} can be required to lie in $H$. \qed

\section{Example: Embeddings into complex-symplectic manifolds}
\label{sec:examples}

In this section, we apply Theorem~\ref{thm:main} to embedding morphisms into
complex-symplectic manifolds. The following assumptions will be maintained
throughout the present section.

\begin{setup}\label{setup:sympl}
  Let $(X,\omega)$ be a complex-symplectic manifold, let $\sF\subset T_X$ be
  the subsheaf of Hamiltonian vector fields, and $f:Y\to X$ the inclusion of a
  compact submanifold.
\end{setup}

\subsection{Infinitesimal deformations locally induced by Hamiltonian  vector fields}

We start off with a discussion of the sheaf of infinitesimal deformations
locally induced by $\sF$. Perhaps somewhat surprisingly, it will turn out that
\emph{all} infinitesimal deformations of $f$ are locally induced by $\sF$, as
long as $Y$ is either a curve, surface, or a Kähler manifold.

\begin{notation}
  If $U \subseteq Y$ is any open set, and $\vec A \in \bigl(f^*T_X\bigr)(U)$,
  consider the associated section $\eta_{\vec A} := \bigl(f^*\omega\bigr)(\vec
  A, \cdot) \in f^*\Omega^1_X$ and the associated form $\xi_{\vec A} :=
  (df)(\eta_{\vec A}) \in \Omega^1_Y$.
\end{notation}

\begin{prop}[Infinitesimal deformations locally induced by Hamiltonians]\label{prop:idliH}
  In Setup~\ref{setup:sympl}, let $\vec A \in H^0\bigl( Y,\, f^*T_X \bigr)$ be
  an infinitesimal deformation of $f$. If $\xi_{\vec A}$ is closed, then $\vec
  A$ is locally induced by $\sF$.
\end{prop}
\begin{proof}
  The question being local on $Y$, we can assume that there are coordinates
  $x_1, \ldots, x_n, y_1, \ldots, y_m$ on $X$ such that the submanifold $Y$ is
  given as $Y = \{ x_1 = \cdots = x_n = 0\}$.  If $\xi_{\vec A}$ is closed, it
  can locally be written as $\xi_{\vec A} = dg$, where $g\in \sO_Y$ is a
  suitable holomorphic function. The section $\eta_{\vec A}$ is then written
  as
  $$
  \eta_{\vec A} = dg + \sum_{i=1}^n g_i(y_1, \ldots, y_m) \cdot dx_i
  $$
  where again $g_i \in \sO_Y$ are suitable holomorphic functions.  Consider
  the function
  $$
  G := g + \sum_{i=1}^n g_i(y_1, \ldots, y_m) x_i,
  $$
  defined on a neighborhood of $Y$ in $X$. To finish the proof, observe that
  the image of its associated Hamiltonian vector field $\Hamilton(G)$ in $f^*
  T_X$ agrees with $\vec A$.
\end{proof}

\begin{cor}[All infinitesimal deformations are locally induced by Hamiltonians]\label{cor:idliH}
  If $Y$ is Kähler or if $\dim Y < 3$, then any infinitesimal deformation of
  $f$ is locally induced by $\sF$. In other words,
  $$
  H^0\big( Y,\, \Image \sF_Y \to f^*T_X \bigr) = H^0\bigl( Y,\, f^*
  T_X \bigr).
  $$
\end{cor}
\begin{proof}
  The proof follows immediately from Proposition~\ref{prop:idliH} and from the
  fact that any holomorphic 1-form on a compact curve, surface or Kähler
  manifolds is automatically closed, cf.~\cite[Chapt.~IV.2]{HBPV} and
  \cite[Thm.~8.28]{Voisin-Hodge1}.
\end{proof}

\subsection{Obstruction sheaf in the symplectic setup}

Next, we show that in the symplectic setting, there exists an obstruction
sheaf that can be understood geometrically. We recall the notation of vector
fields that are perpendicular to $Y$, already introduced in
Example~\ref{ex:symp}.

\begin{notation}
  Let $T_Y^\bot \subseteq f^*T_X$ be the subsheaf of sections which are
  perpendicular to $Y$, i.e., the sheaf associated to the presheaf
  $$
  U\mapsto\{ \vec A\in f^*(T_X)(U)\,\, \bigl|\, \, (f^*\omega)(\vec A,\vec
  B)=0 \text{ for all } \vec B\in T_Y(U) \}
  $$
  where $U\subset Y$ runs over the open subsets. The sheaf $T_Y^\bot$ is a
  subbundle of $f^* T_X$. Its rank equals $\codim_X Y$.
\end{notation}

We aim to show that $T_Y^\bot$ is an obstruction sheaf, in the sense of
Definition~\ref{def:obstrShf}. To start, we need to show that $T_Y^\bot
\subseteq \sF_Y$. In order to prove this, apply the bundle isomorphism
$T_X\cong \Omega^1_X$ induced by $\omega$ to both sides of the inclusion. To
prove $T_Y^\bot \subseteq \sF_Y$ we thus need to show that sections of the
kernel of $df : f^* \Omega^1_X \to \Omega^1_Y$ can locally be extended to
closed forms, defined on a neighborhood of $Y$. In complete analogy to the
proof of Proposition~\ref{prop:idliH}, this follows from a short calculation
in local coordinates which we leave to the reader.

The following proposition then shows the remaining property required for
$T_Y^\bot$ to be an obstruction sheaf.

\begin{prop}[Lie brackets of Hamiltonian vector fields]\label{prop:xkey}
  Let $U \subseteq X$ be any open subset, $Y^\circ := Y \cap U$, and let $\vec
  F, \vec G \in \sF(U)$ be two Hamiltonian vector fields on $U$ that agree
  along $Y^\circ$.  Then the Lie bracket $[\vec F,\vec G]$ is perpendicular to
  $Y$. In other words, $f^* [\vec F,\vec G] \in T_Y^\bot(Y^\circ)$.
\end{prop}
\begin{proof}
  The statement of Proposition~\ref{prop:xkey} being local on $Y$, it suffices
  to show that
  \begin{equation}\label{eq:pedro-x}
    \omega \left( \vec V,[\vec F, \vec G] \right)|_{Y^\circ} = 0
  \end{equation}
  for any vector field $\vec V \in H^0\bigl( U,\, T_X \bigr)$ which is tangent
  to $Y^\circ$, i.e., which satisfies
  $$
  \vec V|_y \in T_Y|_y \subseteq T_X|_y \quad \text{for all }y \in Y^\circ.
  $$
  Given any such $\vec V$, the equality $(\vec F - \vec G)|_{Y^\circ} = 0$
  implies that
  \begin{equation}\label{eq:lola-x}
    \Lie_{\vec F - \vec G}(\vec V)|_{Y^\circ} = 0.
  \end{equation}
  Equality~\eqref{eq:pedro-x} then follows with
  \begin{align*}
    0 & = \Lie_{\vec F-\vec G}\bigl(\omega(\vec V,\vec G)\bigr)|_{Y^\circ} &&
    \text{since $(\vec F-\vec G)|_{Y^\circ} = 0$}\\
    & = \omega(\Lie_{\vec F-\vec G}\vec V,\vec G)|_{Y^\circ} + \omega(\vec V,
    \Lie_{\vec F-\vec G}\vec G)|_{Y^\circ}
    && \text{since $\vec F$ and $\vec G$ preserve $\omega$} \\
    & = \omega(\vec V,[\vec F,\vec G])|_{Y^\circ} && \text{by
      \eqref{eq:lola-x}}
  \end{align*}
  This finishes the proof of Proposition~\ref{prop:xkey}.
\end{proof}

\begin{cor}[Obstruction sheaf in the symplectic setup]\label{cor:symplobst}
  The sheaf $T_Y^\bot\subset \sF_Y$ is an obstruction sheaf for the sheaf
  $\sF\subset T_X$ of Hamiltonian vector fields. \qed
\end{cor}

\subsection{A space of deformations along $\sF$}

As a last step in the discussion of symplectic spaces, we aim to identify a
space of deformations along $\sF$. Since Hamiltonian vector fields preserve
the symplectic form by definition, the following space is a candidate.

\begin{defn}[Space of morphisms with prescribed pull-back of $\omega$]
  Consider subspace
  $$
  \Hom_\omega(Y,X) := \{ g\in \Hom(Y,X) \,|\, g^*(\omega)=f^*(\omega)
  \} \subseteq \Hom(Y,X)
  $$
  with its natural structure as a (not necessarily reduced) complex space.
\end{defn}

The following is now an elementary consequence of the fact that Hamiltonian
vector fields preserve $\omega$.

\begin{fact}[\protect{$\Hom_\omega(Y,X)$ is a space of deformations along $\sF$, cf.~\cite[Prop.~3.1.10]{Joerder10}}]\label{fact:sympldefos}
  If $n$ is any number and $f_n$ any $n$-th order infinitesimal deformation of
  $f$ that is locally induced by time-dependent Hamiltonian vector fields,
  then the corresponding morphism
  $$
  f_n : \Spec\,\bC[t]/t^{n+1} \to \Hom(Y,X)
  $$
  factors via $\Hom_\omega(Y,X)$. In particular, $\Hom_\omega(Y,X)\subseteq
  \Hom(Y,X)$ is a space of deformations along $\sF$. \qed
\end{fact}

\subsection{Summary of results in the symplectic case}

Putting Corollaries~\ref{cor:idliH}, \ref{cor:symplobst},
Fact~\ref{fact:sympldefos} and the main results of this paper,
Theorems~\ref{thm:main} and \ref{thm:mainaddon} together, the following
corollary summarizes our results in the symplectic setting.

\begin{cor}
  Let $f:Y\to X$ be the embedding of a compact curve, surface of Kähler
  submanifold into a complex-symplectic manifold $(X,\omega)$.  Assume that
  $H^1 \bigl(Y,\, T_Y^\bot \bigr) = 0$.  Then any infinitesimal deformation
  $\sigma \in H^0\bigl(Y,\, f^*T_X \bigr)$ is effective.

  More is true: There exists a family $F:\, \Delta\times Y\to X$ of morphisms
  such that the infinitesimal deformation induced by $F$ agrees with $\sigma$,
  i.e., $\sigma_{F,0}=\sigma$, and such that $F_t^*(\omega)=f^*(\omega)$ for
  any $t\in\Delta$.  \qed
\end{cor}

\subsection{Generalisations}

Many of the results contained in this section carry over to the case where $X$
is not necessarily symplectic, but carries a two-tensor $\omega \in H^0 \bigl(
X,\, \Omega^1_X \otimes \Omega^1_X \bigr)$, which need not be alternating,
symmetric, or non-degenerate. Details are found in \cite{Joerder10}.

\section{Example: Deformation along a foliation}
\label{ssec:exDefF}

This subsection is concerned with the case when $X$ carries a regular
foliation, that is, a subbundle $\sF \subseteq T_X$ closed under Lie
bracket. We fix the following setup for the present section.

\begin{setup}\label{setup:foliation}
  Let $\sF \subseteq T_X$ be a regular foliation on a complex manifold $X$,
  and let $f: Y \to X$ be the inclusion of a compact submanifold $Y \subseteq
  X$.
\end{setup}

\subsection{Infinitesimal deformations locally induced by the foliation}

Since $\sF$ is a subbundle of $T_X$, it is clear that in the context of
Setup~\ref{setup:foliation}, the sheaf of infinitesimal deformations locally
induced by $\sF$ is $\sF_Y = f^*\sF$. The space of infinitesimal deformations
locally induced by $\sF$ is then $H^0 \bigl( Y,\, f^* \sF \bigr)$.

\subsection{Obstruction sheaf in the foliated setup}

Of course, the restricted subbundle $\sF_Y\subset f^*(T_X)$ is an obstruction
sheaf for the foliation $\sF$. If the leaves intersect the submanifold $f(Y)$
transversely, then $\sF_Y$ is in fact the only possible obstruction
sheaf. However, if the set
\begin{equation}\label{eq:defT}
  T := \{ y \in Y \,|\, \sF|_y \subseteq T_Y|_y \}
\end{equation}
of points where the foliation is tangent to $Y$ is non-empty, then the
following proposition asserts that the proper subsheaf $\sG:= \sJ_T \otimes
\sF_Y \subset \sF_Y$ is an obstruction sheaf as well.

\begin{prop}\label{cor:folspace}
  In Setup~\ref{setup:foliation}, let $U \subset X$ be any open subset, set
  $Y^\circ := Y \cap U$ and $T^\circ := T \cap U$. If $\vec A, \vec B \in
  \sF(U)$ are two vector fields in $\sF$ that agree along $Y^\circ$, then the
  Lie bracket vanishes along $T^\circ$,
  $$
  f^* \bigl( [\vec A, \vec B] \bigr)|_{T^\circ} = 0 \in \sF_Y
  (T^\circ).
  $$
\end{prop}
\begin{proof}
  We consider vector fields as derivations acting on the structure sheaf
  $\sO_X$. From this point of view, we need to show $\bigl( [\vec A,\vec
  B].g\bigr) (x)=0$ for any point $x \in T^\circ$ and any germ of function $g
  \in \sO_{X,x}$. We know that
  $$
  [\vec A, \vec B].g =[\vec A,\vec B-\vec A].g = \vec
  A.\bigl(\underbrace{(\vec B-\vec A).g}_{=: a} \bigr) - \underbrace{(\vec
    B-\vec A).\bigl(\vec A.g \bigr)}_{=: b},
  $$
  The terms $a$ and $b$ vanish along $Y^\circ$ because $\vec A|_{Y^\circ} =
  \vec B|_{Y^\circ}$. Since $\vec A \in \sF$ is tangent to $Y$ at all points
  $x \in T^\circ$, the assertion follows.
\end{proof}

\begin{cor}[Obstruction sheaf in the case of a foliation]\label{cor:folobst}
  In Setup~\ref{setup:foliation}, any sheaf $\sG$ satisfying
  $$
  \sJ_T \otimes \sF|_Y \subseteq \sG \subseteq \sF|_Y
  $$
  is an obstruction sheaf for the foliation $\sF$. \qed
\end{cor}

\subsection{A space of deformations along $\sF$}

A space of deformations along $\sF$ has been constructed in \cite{KKL10}. The
following notation is useful in the description of its main property.

\begin{notation}[\protect{Velocity vector field for families of morphisms, cf.~\cite[Sect.~1.B]{KKL10}}]\label{not:infdef}
  Let $F : \Delta \times Y \to X$ be a family of morphisms such that $F_0 =
  f$. Given a point $y \in Y$, we can consider the curve
  $$
  F_y : \Delta \to X, \quad t \mapsto F(t,y).
  $$
  Given $t_0 \in \Delta$ and taking derivatives in $t$ for all $y$ at time
  $t=t_0$, this gives a section
  $$
  \sigma_{F, t_0} \in H^0 \bigl( Y,\, (F_{t_0})^* T_X \bigr),
  $$
  called \emph{velocity vector field at time $t_0$}.
\end{notation}

\begin{fact}[\protect{Space of deformations along a foliation, cf.~proof of~\cite[Cor.~5.6]{KKL10}}]\label{fact:folspace}
  In Setup~\ref{setup:foliation}, there exists a space of deformations along
  $\sF$, denoted
  $$
  \Hom_\sF(Y,X) \subseteq \Hom(Y,X),
  $$
  with the following additional property.  If $F: \Delta \to \Hom_\sF(Y,X)$ is
  any holomorphic curve germ, with associated family of morphisms $F: \Delta
  \times Y \to X$, then the velocity vector fields are in the pull-back of the
  foliation $\sF$, for all times $t \in \Delta$. In other words, $\sigma_{F,t}
  \in H^0\bigl(Y ,\, F_t^* \sF \bigr)$ for all $t \in \Delta$.  \qed
\end{fact}

The actual statement of~\cite[Cor.~5.6]{KKL10} only implies that infinitesimal
deformations induced by \emph{time-independent} vector fields in $\sF$ factor
through $Hom_\sF(Y,X)$. However, the proof in \emph{loc.~cit.{}} can be
generalized with minimal changes to the case of \emph{time-dependent} vector
fields, as required in Definition~\ref{defn:spaceofdefo}.

\subsection{Summary of results in the case of a foliated manifold}

Using Corollary~\ref{cor:folobst} and Fact~\ref{fact:folspace}, the following
corollary summarizes our results in the case of a foliated manifolds.

\begin{cor}\label{cor:dalf}
  In Setup~\ref{setup:foliation}, let $\sG$ be any subsheaf satisfying
  $$
  \sJ_T \otimes \sF_Y \subseteq \sG \subseteq \sF|_Y,
  $$
  where $T \subseteq Y$ is the space defined in \eqref{eq:defT} above.  Assume
  that $H^1 \bigl(Y,\,\sG \bigr) = 0$.
  
  Then any infinitesimal deformation $\sigma \in H^0 \bigl(Y,\, \sF|_Y \bigr)$
  is effective. More is true: there exists a family $F:\Delta\times Y\to X$ of
  morphisms such that $F_0=f$, $\sigma_{F,0}=\sigma$ and, such that the
  velocity vector fields $\sigma_{F,t}$ are contained in $H^0\bigl(Y ,\, F_t^*
  \sF \bigr)$, for all $t \in \Delta$. \qed
\end{cor}

\begin{rem}
  It might be worth noting that to obtain the conclusions of
  Corollary~\ref{cor:dalf}, it suffices to prove vanishing $H^1 \bigl(Y,\,\sG
  \bigr) = 0$ for a single obstruction sheaf $\sG$.  Since the obstruction
  sheaf is often not uniquely defined, this gives extra leeway which might be
  useful in applications.
\end{rem}

\providecommand{\bysame}{\leavevmode\hbox to3em{\hrulefill}\thinspace}
\providecommand{\MR}{\relax\ifhmode\unskip\space\fi MR}
\providecommand{\MRhref}[2]{%
  \href{http://www.ams.org/mathscinet-getitem?mr=#1}{#2}
}
\providecommand{\href}[2]{#2}

\end{document}